\documentclass[12pt]{amsart}
\usepackage{amsmath}
\usepackage{amssymb}
\usepackage[abbrev]{amsrefs}
\usepackage{mathtools}
\usepackage{longtable}
\usepackage{bm}
\usepackage{sepnum}
\usepackage{cleveref}
\usepackage{tikz}

\newtheorem{theorem}{Theorem}[section]
\newtheorem{lemma}[theorem]{Lemma}
\newtheorem{definition}[theorem]{Definition}
\newtheorem{remark}[theorem]{Remark}
\newtheorem{conjecture}[theorem]{Conjecture}
\newtheorem{example}[theorem]{Example}

\newcommand{\Ann}{\operatorname{Ann}}
\newcommand{\SLP}{\mathrm{SLP}}
\newcommand{\semicolon}{\mathbin{;}}

\begin{document}

\title{Failure of the Lefschetz property for the Graphic Matroid}
\author{Ryo Takahashi}
\address{Mathematical Institute, Graduate School of Science, Tohoku University, Sendai, Japan}
\date{\today}

\begin{abstract}
    We consider the strong Lefschetz property for standard graded Artinian Gorenstein algebras. Such an algebra has a presentation of the quotient algebra of the ring of the differential polynomials modulo the annihilator of some homogeneous polynomial. There is a characterization of the strong Lefschetz property for such an algebra by the non-degeneracy of the higher Hessian matrix of the homogeneous polynomial.
    Maeno and Numata conjectured that if such an algebra is defined by the basis generating polynomial of any matroid, then it has the strong Lefschetz property.
    For this conjecture, we give counterexamples that are associated with graphic matroids. We prove the degeneracy of the higher Hessian matrix by constructing a non-zero element in the kernel of that matrix.
\end{abstract}

\maketitle

\section{Introduction}\label{section:introduction}

The Lefschetz property for Artinian Gorenstein algebras is inspired by the Hard Lefschetz Theorem on the cohomology of smooth complex projective varieties.

Let $n$ be a positive integer and $\mathbb{K}$ a field of characteristic zero. The polynomial algebra $\mathbb{K}[x_1,\dots,x_n]$ is regarded as a module over the algebra $Q\coloneqq\mathbb{K}[\partial_1,\dots,\partial_n]$ where $\partial_i\coloneqq\frac{\partial}{\partial x_i}$ for $i=1,\dots,n$.

For a homogeneous polynomial $f\in\mathbb{K}[x_1,\dots,x_n]$ of degree $d$, let
\begin{align*}
    \Ann_Q(f)&\coloneqq\{\alpha\in Q\mid\alpha f=0\},\\
    A&\coloneqq Q/\!\Ann_Q(f).
\end{align*}

Since $f$ is homogeneous, $\Ann_Q(f)$ is a homogeneous ideal of $Q$. Thus the algebra $A$ can be decomposed into
\[A=\bigoplus_{i=0}^d A_i\]
as a graded Artinian algebra. Note that for all $i>d$, the homogeneous part of degree $i$ of $A$ is equal to $\{0\}$.

Furthermore, the algebra $A$ is a \emph{Poincar\'{e} duality algebra}. In other words, $A_d$ is congruent to the ground field $\mathbb{K}$ of characteristic zero; and the bilinear pairing
\[A_i\times A_{d-i}\rightarrow A_d\]
is non-degenerate for all $i=0,\dots,d$. It is known that a graded Artinian algebra is Gorenstein if and only if it is a Poincar\'{e} duality algebra (see~\cite{harima2013lefschetz}*{Theorem~2.79}). Hence, the algebra $A$ is Gorenstein. The number $d$ is called the \emph{socle degree} of $A$. Conversely, according to~\cite{harima2013lefschetz}*{Lemma~3.74}, any standard graded Artinian Gorenstein algebra has a presentation $Q/\!\Ann_Q(f)$ with some homogeneous polynomial $f$.

We say that the algebra $A$ has the \emph{strong Lefschetz property} if there exists an element $\ell\in A_1$ such that the multiplication map
\[\times\ell^k\colon A_i\rightarrow A_{i+k}\]
has full rank for all $i=0,\dots,d-1$ and $k=1,\dots,d-i$. A weakening of this definition to only $k=1$ is called the \emph{weak Lefschetz property}.

\subsubsection*{Conjecture}
In the papers~\cites{maeno2012sperner,maeno2016sperner}, Maeno and Numata conjectured that if $f$ is a basis generating polynomial of any matroid, then the algebra $A$ has the strong Lefschetz property. They showed this conjecture for certain matroids to prove the Sperner property of modular geometric lattices.

\subsubsection*{Previous work}
The strong Lefschetz property at $i=1$ is studied in \cites{murai2021strictness,kirchhoff,yazawa2021eigenvalues} as it is related to the Hessian matrix of the polynomial $f$. As the most general result, Murai, Nagaoka, and Yazawa~\cite{murai2021strictness}*{Theorem~3.8~and~Remark~3.9} showed that under the conjecture's condition (i.e., $f$ is a basis generating polynomial of any matroid), the multiplication map
\[\times\ell^k\colon A_1\rightarrow A_{1+k}\]
has full rank for all $k=1,\dots,d-1$, where $\ell=a_1\partial_1+\dots+a_n\partial_n\in A_1$ for any $(a_1,\dots,a_n)\in\mathbb{R}_{>0}^n$.

\subsubsection*{Our contributions}
We try to verify this conjecture by computation with the mathematical software system \texttt{SageMath}~\cite{sagemath}. We concentrate on a class of matroids called \emph{graphic matroids}. For details of graphs and matroids, see~\cite{oxleymatroidtheory}. By employing an enumeration of all simple graphs, we construct graphic matroids of them and check whether the conjecture holds.

Although the conjecture is true for all graphs of seven or fewer vertices, we found counterexamples associated with graphs of eight vertices. One of them (see \Cref{figure:SLP_3 in introduction}), the algebra $A=Q/\!\Ann_Q(f)$, with the smallest \emph{codimension} $\dim_{\mathbb{R}}A_1$ has the following characteristic values: the number of variables $n=13$, the socle degree $d=7$, the minimal number of generators $\mu(\Ann_Q(f))=69$, and the \emph{Hilbert function} $(\dim_{\mathbb{R}}A_i)_{i=0}^d=(1,13,70,166,166,70,13,1)$. The algebra $A$ does not have the strong Lefschetz property at $i=3$, i.e., there is no element $\ell\in A_1$ such that the multiplication map
\[\times\ell\colon A_3\rightarrow A_4\]
has full rank. It means that the algebra $A$ does not even have the weak Lefschetz property.

Whether there exists a case that fails at $i=2$ remains unknown. For a setting without any constraints on the polynomial, refer to \Cref{example: ikeda's example}.

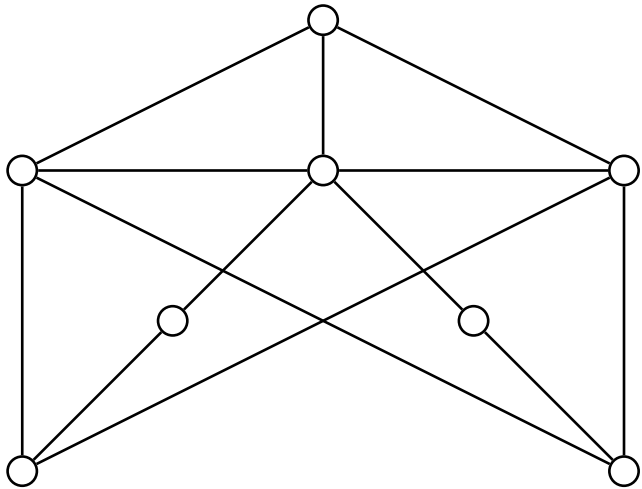
\begin{figure}[tb]
    \centering
    \begin{tikzpicture}
        \node[draw,circle,line width=1pt](5) at (0,2){};
        \node[draw,circle,line width=1pt](2) at (4,0){};
        \node[draw,circle,line width=1pt](8) at (0,0){};
        \node[draw,circle,line width=1pt](1) at (-4,0){};
        \node[draw,circle,line width=1pt](4) at (2,-2){};
        \node[draw,circle,line width=1pt](7) at (4,-4){};
        \node[draw,circle,line width=1pt](3) at (-2,-2){};
        \node[draw,circle,line width=1pt](6) at (-4,-4){};

        \draw[line width=1pt](1)--(5);
        \draw[line width=1pt](1)--(6);
        \draw[line width=1pt](1)--(7);
        \draw[line width=1pt](1)--(8);
        \draw[line width=1pt](2)--(5);
        \draw[line width=1pt](2)--(6);
        \draw[line width=1pt](2)--(7);
        \draw[line width=1pt](2)--(8);
        \draw[line width=1pt](3)--(6);
        \draw[line width=1pt](3)--(8);
        \draw[line width=1pt](4)--(7);
        \draw[line width=1pt](4)--(8);
        \draw[line width=1pt](5)--(8);
    \end{tikzpicture}
    \caption{A graph corresponding to one of the counterexamples}
    \label{figure:SLP_3 in introduction}
\end{figure}

\subsubsection*{Organization}
The rest of this paper is organized as follows. \Cref{section:preliminaries} contains detailed settings of the strong Lefschetz property and the conjecture. \Cref{section:SLP_3} describes our computation, the failure of the strong Lefschetz property at $i=3$, and the details of counterexamples.

Finally, we note the partial failure of the strong Lefschetz property at $i=2$ in \Cref{section:SLP_2}. In contrast to the previous work, the element $\ell=\partial_1+\dots+\partial_n\in A_1$ is not a universal solution, despite the fact that $(1,\dots,1)\in\mathbb{R}_{>0}^n$.

\section{Preliminaries}\label{section:preliminaries}

\Cref{subsection:strong lefschetz property} contains the details of the strong Lefschetz property for graded Artinian Gorenstein algebras. The conjecture is in \Cref{subsection:basis generating polynomial}.

\subsection{Strong Lefschetz Property}\label{subsection:strong lefschetz property}

We assume that $\mathbb{K}=\mathbb{R}$. Let $f\in\mathbb{R}[x_1,\dots,x_n]$ be a homogeneous polynomial of degree $d$. The algebra
\[A=A(f)\coloneqq Q/\!\Ann_Q(f)\]
is a standard graded Artinian Gorenstein algebra and can be decomposed into
\[A=\bigoplus_{i=0}^d A_i.\]
Hereafter, we represent elements of $A$ in terms of elements of $Q$ without causing ambiguity.

The Poincar\'{e} duality algebra $A$ satisfies the following properties:
\begin{itemize}
    \item The linear map
    \[[\bullet]\colon A_d\rightarrow\mathbb{R}\semicolon\quad[\alpha]\coloneqq\alpha f\]
    is the isomorphism from $A_d$ to $\mathbb{R}$.
    \item For each $i=0,\dots,d$, the bilinear form
    \[A_i\times A_{d-i}\rightarrow\mathbb{R}\semicolon\quad(\xi,\eta)\mapsto[\xi\eta]\]
    is non-degenerate.
\end{itemize}

Next, we define the strong Lefschetz property of the algebra $A$.
\begin{definition}[strong Lefschetz property (in the narrow sense)]
    Let $k\le d/2$ be a non-negative integer. We say that $A$ has the \emph{strong Lefschetz property at degree $k$}, shortly $\SLP_k$, if there exists an element $\ell\in A_1$ such that the multiplication map
    \[\times\ell^{d-2k}\colon A_k\rightarrow A_{d-k}\]
    is an isomorphism. In addition, if $A$ has the $\SLP_k$ for all $k=0,\dots,\lfloor d/2\rfloor$ with a common element $\ell\in A_1$, we say that $A$ has the \emph{strong Lefschetz property} and $\ell$ is a \emph{Lefschetz element}.
\end{definition}
\begin{remark}
    For the Poincar\'{e} duality algebra $A$, this definition of the strong Lefschetz property is equivalent to the one given in \Cref{section:introduction}. Specifically, for any $i=0,\dots,d-1$ and $k=1,\dots,d-i$, the following implications hold:
    \begin{itemize}
        \item If $i+k\le d-i$, then the composition
        \[A_i\xrightarrow{\times\ell^k}A_{i+k}\xrightarrow{\times\ell^{d-2i-k}}A_{d-i}\]
        is bijective, which implies that $A_i\xrightarrow{\times\ell^k}A_{i+k}$ is injective.
        \item If $i+k> d-i$, then the composition
        \[A_{d-i-k}\xrightarrow{\times\ell^{2i+k-d}}A_i\xrightarrow{\times\ell^k}A_{i+k}\]
        is bijective, which implies that $A_i\xrightarrow{\times\ell^k}A_{i+k}$ is surjective.
    \end{itemize}
\end{remark}

To test the strong Lefschetz property, we employ the higher Hessian matrix.

\begin{definition}[higher Hessian matrix]
    Let $k$ be a non-negative integer and $B_k=\{\alpha_1,\dots,\alpha_m\}$ be a set of homogeneous polynomials of degree $k$ in $Q$.
    For polynomial $g\in\mathbb{R}[x_1,\dots,x_n]$, we define an $m\times m$ polynomial matrix $\bm{H}_{B_k}(g)$ by
    \[\left(\bm{H}_{B_k}(g)\right)_{i,j}\coloneqq(\alpha_i\alpha_j)g\quad(i,j=1,\dots,m).\]
    This $\bm{H}_{B_k}(g)$ is called the \emph{$k$-th Hessian matrix} of $g$ with respect to $B_k$.
\end{definition}
When $B_1=\{\partial_1,\dots,\partial_n\}$, the first Hessian matrix $\bm{H}_{B_1}(g)$ coincides with the usual Hessian matrix of $g$.

The strong Lefschetz property of the algebra $A$ can be examined using this matrix.

\begin{theorem}[\cite{harima2013lefschetz}*{Theorem~3.76},\cite{MR2594646}*{Theorem~3.1},\cite{watanabe2000hessian}*{Theorem~4}]\label{theorem:SLP check}
    Let $k\le d/2$ be a non-negative integer and $B_k=\{\alpha_1,\dots,\alpha_m\}$ be any $\mathbb{R}$-basis of $A_k$.
    For any $(a_1,\dots,a_n)\in\mathbb{R}^n$, the algebra $A$ has the $\SLP_k$ with an element $\ell=a_1\partial_1+\dots+a_n\partial_n\in A_1$ if and only if $\bm{H}_{B_k}(f)(a_1,\dots,a_n)$ is non-degenerate where
    \[(\bm{H}_{B_k}(f)(a_1,\dots,a_n))_{i,j}\coloneqq((\alpha_i\alpha_j)f)(a_1,\dots,a_n)\quad(i,j=1,\dots,m).\]
\end{theorem}
\begin{proof}
    Since the algebra $A$ is a Poincar\'{e} duality algebra, $A$ has the $\SLP_k$ with an element $\ell$ if and only if a bilinear form
    \[A_k\times A_k\rightarrow\mathbb{R}\semicolon\quad(\xi,\eta)\mapsto[\ell^{d-2k}\xi\eta]\]
    is non-degenerate. The representation matrix of this bilinear form with respect to the $\mathbb{R}$-basis $B_k$ of $A_k$ has the $(i,j)$-entry
    \[[\ell^{d-2k}\alpha_i\alpha_j]=(\ell^{d-2k}\alpha_i\alpha_j)f=(d-2k)!\,((\alpha_i\alpha_j)f)(a_1,\dots,a_n)\quad(i,j=1,\dots,m).\]
    The last equality is due to Euler's homogeneous function theorem. Thus the representation matrix is $(d-2k)!\,\bm{H}_{B_k}(f)(a_1,\dots,a_n)$.
\end{proof}

\begin{remark}
    The algebra $A$ has the $\SLP_k$ if and only if the polynomial $\det\bm{H}_{B_k}(f)$ is non-zero. Furthermore, if $A$ has the $\SLP_k$ for all $k=0,\dots,\lfloor d/2\rfloor$, then there exists a common non-root $(a_1,\dots,a_n)$ of the polynomials $\det\bm{H}_{B_k}(f)$. In this case, the algebra $A$ has the strong Lefschetz property with the Lefschetz element $\ell=a_1\partial_1+\dots+a_n\partial_n$.
\end{remark}

We demonstrate the degeneracy of the higher Hessian matrix.
\begin{example}[\cite{ikeda1996results}*{Example~4.4}]\label{example: ikeda's example}
    Let $f\coloneqq x_1^3x_2x_3+x_1x_2^3x_4+x_3^3x_4^2\in\mathbb{R}[x_1,\dots,x_4]$. The set of polynomials
    \[B_2\coloneqq\{\partial_1^2,\partial_1\partial_2,\partial_1\partial_3,\partial_1\partial_4,\partial_2^2,\partial_2\partial_3,\partial_2\partial_4,\partial_3^2,\partial_3\partial_4,\partial_4^2\}\subset A_2\]
    forms an $\mathbb{R}$-basis of $A_2$. The corresponding higher Hessian matrix is given by
    \[\bm{H}_{B_2}(f)=\begin{pmatrix}
        0 & 6 x_3 & 6 x_2 & 0 & 0 & 6 x_1 & 0 & 0 & 0 & 0 \\
        6 x_3 & 0 & 6 x_1 & 0 & 6 x_4 & 0 & 6 x_2 & 0 & 0 & 0 \\
        6 x_2 & 6 x_1 & 0 & 0 & 0 & 0 & 0 & 0 & 0 & 0 \\
        0 & 0 & 0 & 0 & 6 x_2 & 0 & 0 & 0 & 0 & 0 \\
        0 & 6 x_4 & 0 & 6 x_2 & 0 & 0 & 6 x_1 & 0 & 0 & 0 \\
        6 x_1 & 0 & 0 & 0 & 0 & 0 & 0 & 0 & 0 & 0 \\
        0 & 6 x_2 & 0 & 0 & 6 x_1 & 0 & 0 & 0 & 0 & 0 \\
        0 & 0 & 0 & 0 & 0 & 0 & 0 & 0 & 12 x_4 & 12 x_3 \\
        0 & 0 & 0 & 0 & 0 & 0 & 0 & 12 x_4 & 12 x_3 & 0 \\
        0 & 0 & 0 & 0 & 0 & 0 & 0 & 12 x_3 & 0 & 0
    \end{pmatrix}.\]
    By calculating the determinant we confirm that $\det\bm{H}_{B_2}(f)=0$, which verifies the degeneracy. Additionally, this result is observed through the equation
    \[\bm{H}_{B_2}(f)\times(0,0,x_1x_2^2,x_1^3,0,-x_2^3,-x_1^2x_2,0,0,0)^\mathrm{T}=\bm{0}.\]

    Consequently, the algebra $A$ does not have the $\SLP_2$. The Hilbert function of $A$ is $(1,4,10,10,4,1)$.
\end{example}

\subsection{Basis Generating Polynomial}\label{subsection:basis generating polynomial}

Let $G$ be a connected graph with $d+1$ vertices and $n$ edges. We number the edges one through $n$ and identify the edges with the numbers.

A subgraph $T$ of the graph $G$ is called a \emph{spanning tree} of $G$ if $T$ is connected graph on the same vertices as $G$ without cycles.
\begin{definition}[basis generating polynomial (for graphs)]
    The \emph{basis generating polynomial} $f_G\in\mathbb{R}[x_1,\dots,x_n]$ of the graph $G$ is defined as the sum, over all spanning trees $T$ of $G$, of the products of $x_e$ for the edges $e$ in $T$. More formally,
    \[f_G\coloneqq\sum_{T}\prod_{e}x_e.\]
\end{definition}
A spanning tree of $G$ is a \emph{basis} of the graphic matroid of $G$. The name \emph{basis} comes from the term of matroids. Since the number of edges in every spanning tree of $G$ is $d$, the polynomial $f_G$ is a homogeneous polynomial of degree $d$.

From the above, the conjecture we mentioned in \Cref{section:introduction} is as follows.
\begin{conjecture}[Maeno--Numata conjecture (for graphs)~\cite{maeno2012sperner}*{conjecture}]\label{conjecture}
    The algebra $A(f_G)$ has the strong Lefschetz property for any connected graph $G$.
\end{conjecture}

\begin{remark}
    If edges $i$ and $j$ are multiple edges, then $\partial_i-\partial_j\in\Ann_Q(f_G)$. If an edge $k$ is a self loop, then $\partial_k\in\Ann_Q(f_G)$. Thus for \Cref{conjecture} we can ignore self loops and multiple edges; and focus only on \emph{simple} graphs, i.e., graphs without such edges.
\end{remark}

In the following, we say that the graph $G$ has the strong Lefschetz property or the $\SLP_k$ if the algebra $A(f_G)$ has the strong Lefschetz property or the $\SLP_k$, respectively.

\section{Failure of the $\SLP_3$}\label{section:SLP_3}

In this section, we provide a planar graph and a non-planar graph without the $\SLP_3$ as the counterexamples to \Cref{conjecture}. Our probabilistic and deterministic methods are described in \Cref{subsection:screening,subsection:verification}, respectively. In addition, we included the list of counterexample candidates in \Cref{subsection:list of counterexample candidates}.

\subsection{Screening}\label{subsection:screening}

To prove that the algebra $A=A(f_G)$ does not have the $\SLP_3$, we need to confirm that the third Hessian matrix $\bm{H}_{B_3}(f_G)$ is degenerate for an $\mathbb{R}$-basis $B_3$ of $A_3$. However, it is hard to compute the determinant of that large matrix of multivariate polynomials. For this reason, we first employ a randomized algorithm due to \Cref{lemma:schwartz-zippel}.

\begin{lemma}[Schwartz–-Zippel lemma~\cites{DEMILLO1978193,schwartz,zippel}]\label{lemma:schwartz-zippel}
    Let $g\in\mathbb{R}[x_1,\dots,x_n]$ be a non-zero polynomial. Suppose that $S$ is a finite subset of $\mathbb{R}$ and $r_1,\dots,r_n$ are selected at random independently and uniformly from $S$. Then,
    \[\operatorname{Pr}[g(r_1,\dots,r_n)=0]\le\frac{\deg g}{|S|}.\]
\end{lemma}

Since the $\SLP_3$ is trivial or undefined for graphs of seven or fewer vertices, let $G$ be a graph of eight vertices. Every entry of $\bm{H}_{B_3}(f_G)$ is the sixth-order partial derivative of the polynomial $f_G$ of degree seven. Thus if the polynomial $g\coloneqq \det \bm{H}_{B_3}(f_G)$ is non-zero, then $\deg g=\dim_{\mathbb{R}}A_3$.

We repeated the following check 100 times: select $r_1,\dots,r_n$ at random independently and uniformly from the set $S=\{1,\dots,10^9\}$ and assure whether $g(r_1,\dots,r_n)=0$. If the polynomial $g$ is non-zero, then the probability that $g$ passes our check is less than $\left(\frac{\dim_{\mathbb{R}}A_3}{|S|}\right)^{100}$. Because our computation showed $\dim_{\mathbb{R}}A_3\le 500$, this probability is smaller than ${10}^{-630}$.

We found $152$ counterexample candidates out of $\sepnum{}{,}{}{11117}$ simple connected graphs of eight vertices, up to isomorphism of graphs. They are listed in \Cref{subsection:list of counterexample candidates}. Among them, for the one which has the smallest number of edges (\Cref{figure:SLP_3}, the same graph as \Cref{figure:SLP_3 in introduction} in \Cref{section:introduction}) and one of which is a planar graph (\Cref{figure:SLP_3 planar}), we verify that $\bm{H}_{B_3}(f_G)$ is degenerate for an $\mathbb{R}$-basis $B_3$ of $A_3$. The details of \Cref{figure:SLP_3} are in the next section.
\begin{remark}\label{remark: biconnected components}
    We only need to check \emph{biconnected graphs}, namely graphs whose connectivity is preserved when any one vertex is deleted. The reason is as follows. First, the polynomial $f_G$ is a product of the basis generating polynomials of each \emph{biconnected components} of $G$, that are maximal biconnected subgraphs of $G$. In consequence, according to \cite{harima2013lefschetz}*{Theorem~3.34~and~Proposition~3.77}, the strong Lefschetz property of each biconnected component derives the strong Lefschetz property of the whole graph $G$.

    The number of biconnected graphs of eight vertices is $\sepnum{}{,}{}{7123}$, up to isomorphism.
\end{remark}

\begin{figure}[tb]
    \begin{minipage}{.49\textwidth}
        \centering
        \begin{tikzpicture}[scale=0.7]
            \node[draw,circle,line width=1pt](5) at (0,2){};
            \node[draw,circle,line width=1pt](2) at (4,0){};
            \node[draw,circle,line width=1pt](8) at (0,0){};
            \node[draw,circle,line width=1pt](1) at (-4,0){};
            \node[draw,circle,line width=1pt](4) at (2,-2){};
            \node[draw,circle,line width=1pt](7) at (4,-4){};
            \node[draw,circle,line width=1pt](3) at (-2,-2){};
            \node[draw,circle,line width=1pt](6) at (-4,-4){};
    
            \draw[line width=1pt](1)--(5)node[pos=0.5,above]{$1$};
            \draw[line width=1pt](1)--(6)node[pos=0.5,left]{$2$};
            \draw[line width=1pt](1)--(7)node[pos=0.7,below]{$3$};
            \draw[line width=1pt](1)--(8)node[pos=0.5,above]{$4$};
            \draw[line width=1pt](2)--(5)node[pos=0.5,above]{$5$};
            \draw[line width=1pt](2)--(6)node[pos=0.7,below]{$6$};
            \draw[line width=1pt](2)--(7)node[pos=0.5,right]{$7$};
            \draw[line width=1pt](2)--(8)node[pos=0.5,above]{$8$};
            \draw[line width=1pt](3)--(6)node[pos=0.7,above]{$9$};
            \draw[line width=1pt](3)--(8)node[pos=0.3,above]{$10$};
            \draw[line width=1pt](4)--(7)node[pos=0.7,above]{$11$};
            \draw[line width=1pt](4)--(8)node[pos=0.3,above]{$12$};
            \draw[line width=1pt](5)--(8)node[pos=0.5,right]{$13$};
        \end{tikzpicture}
        \caption{The graph with the smallest number of edges without $\SLP_3$}
        \label{figure:SLP_3}
    \end{minipage}
    \begin{minipage}{.49\textwidth}
        \centering
        \begin{tikzpicture}[scale=0.7]
            \node[draw,circle,line width=1pt](4) at (0,0){};
            \node[draw,circle,line width=1pt](6) at (0,2){};
            \node[draw,circle,line width=1pt](7) at (-4,0){};
            \node[draw,circle,line width=1pt](8) at (4,0){};
            \node[draw,circle,line width=1pt](1) at (-2,2){};
            \node[draw,circle,line width=1pt](2) at (2,2){};
            \node[draw,circle,line width=1pt](5) at (0,4){};
            \node[draw,circle,line width=1pt](3) at (0,6){};
    
            \draw[line width=1pt](1)--(5);
            \draw[line width=1pt](1)--(6);
            \draw[line width=1pt](1)--(7);
            \draw[line width=1pt](2)--(5);
            \draw[line width=1pt](2)--(6);
            \draw[line width=1pt](2)--(8);
            \draw[line width=1pt](3)--(5);
            \draw[line width=1pt](3)to[out=180,in=90](7);
            \draw[line width=1pt](3)to[out=0,in=90](8);
            \draw[line width=1pt](4)--(6);
            \draw[line width=1pt](4)--(7);
            \draw[line width=1pt](4)--(8);
            \draw[line width=1pt](5)to[out=180,in=70](7);
            \draw[line width=1pt](5)to[out=0,in=110](8);
            \draw[line width=1pt](6)--(7);
            \draw[line width=1pt](6)--(8);
            \draw[line width=1pt](7)to[out=-30,in=-150](8);
        \end{tikzpicture}
        \caption{A planar graph without $\SLP_3$}
        \label{figure:SLP_3 planar}
    \end{minipage}
\end{figure}

\subsection{Verification}\label{subsection:verification}

Let $G$ be the graph shown in \Cref{figure:SLP_3}. The numbering of each edge is also shown in \Cref{figure:SLP_3}. The algebra $A=A(f_G)$ is mentioned in \Cref{section:introduction}: the number of variables $n=13$, the socle degree $d=7$, the minimal number of generators $\mu(\Ann_Q(f_G))=69$, and the Hilbert function $(\dim_{\mathbb{R}}A_i)_{i=0}^d=(1,13,70,166,166,70,13,1)$. Let $m\coloneqq\dim_{\mathbb{R}}A_3=166$.

We fix the $\mathbb{R}$-basis $B_3$ of $A_3$ as follows. Let $(i,j,k)\coloneqq\partial_i\partial_j\partial_k$ for $1\le i<j<k\le n$. The set of monomials $\{(i,j,k)\mid 1\le i<j<k\le n\}$ is a generating set of $A_3$ because the polynomial $f_G$ is a square-free polynomial. Enumerate this set in the lexicographic order, as $\beta_1=(1,2,3),\dots,\beta_{\binom{n}{3}}=(n-2,n-1,n)$. We use
\[B_3\coloneqq\left\{\beta_i\;\middle|\;1\le i\le\binom{n}{3},\langle\beta_1,\dots,\beta_{i-1}\rangle\subsetneq\langle \beta_1,\dots,\beta_i\rangle\right\}.\]

The third Hessian matrix $\bm{H}_{B_3}(f_G)$ contains $\sepnum{}{,}{}{8450}$ non-zero entries. To verify the degeneracy of $\bm{H}_{B_3}(f_G)$, we construct a non-zero vector of polynomials $\bm F=(F_1,\dots,F_m)^\mathrm{T}\in\mathbb{R}[x_1,\dots,x_n]^m$ such that $\bm{H}_{B_3}(f_G)\bm F=\bm{0}$. Such an $\bm{F}$ satisfies the following conditions, although these conditions will not characterize $\bm{F}$:
\begin{theorem}
    Let $k=0,\dots,\lfloor d/2\rfloor$, $B_k=\{\alpha_1,\dots,\alpha_m\}$ be any $\mathbb{R}$-basis of $A_k$, and $\bm F\coloneqq(F_1,\dots,F_m)^\mathrm{T}\in\mathbb{R}[x_1,\dots,x_n]^m$ with $\bm{H}_{B_k}(f)\bm F=\bm 0$. The following hold.
    \begin{enumerate}
        \item For any $(a_1,\dots,a_n)\in\mathbb{R}^n$, let $\ell\coloneqq a_1\partial_1+\dots+a_n\partial_n$ and
        \[\xi\coloneqq\sum_{i=1}^mF_i(a_1,\dots,a_n)\alpha_i.\]
        The element $\xi\in A_k$ is in the kernel of the multiplication map $\times\ell^{d-2k}\colon A_k\rightarrow A_{d-k}$.
        \item
        \[\sum_{i=1}^m F_i\cdot(\alpha_i f)=0.\]
    \end{enumerate}
\end{theorem}
\begin{proof}
\begin{enumerate}
    \item It suffices to show that for all $\eta\in A_k$, $[\ell^{d-2k}\xi\eta]=0$ because of the Poincar\'{e} duality of $A$. According to the proof of \Cref{theorem:SLP check}, a linear form $\eta\mapsto[\ell^{d-2k}\xi\eta]$ is represented with respect to the $\mathbb{R}$-basis $B_k$ of $A_k$ by
    \[(d-2k)!\bm{H}_{B_k}(f)(a_1,\dots,a_n)\bm{F}(a_1,\dots,a_n)=(d-2k)!(\bm{H}_{B_k}(f)\bm{F})(a_1,\dots,a_n)=\bm 0.\]
    \item It suffices to show that for all $(a_1,\dots,a_n)\in\mathbb{R}^n$,
    \[\left(\sum_{i=1}^m F_i\cdot(\alpha_i f)\right)(a_1,\dots,a_n)=0.\]
    Let $\ell$ and $\xi$ be as in (1), then
    \begin{align*}
        \left(\sum_{i=1}^m F_i\cdot(\alpha_i f)\right)(a_1,\dots,a_n)&=\left(\sum_{i=1}^m F_i(a_1,\dots,a_n)\cdot(\alpha_i f)\right)(a_1,\dots,a_n)\\
        &=(\xi f)(a_1,\dots,a_n)=\frac{(\ell^{d-k}\xi)f}{(d-k)!}=0
    \end{align*}
    because according to (1), $\ell^{d-2k}\xi\in\Ann_Q(f)$.
\end{enumerate}
\end{proof}

\subsubsection{Construction Algorithm of $\bm F$}
The vector of polynomials $\bm F$ is constructed by the following steps with the \emph{polynomial interpolation}.
\begin{enumerate}
    \item Find the maximum degree
    \[D_i\coloneqq\max_{j=1,\dots,m}\deg F_j(1,\dots,x_i,\dots,1),\]
    where $F_j(1,\dots,x_i,\dots,1)$ is the univariate polynomial in $x_i$ obtained by substituting one for $x_k$ other than $x_i$ in $F_j$.
    \item Construct the polynomials $\bm F$ by the multivariate polynomial interpolation from the values of $\bm{F}$ at each point $(a_1,\dots,a_n)\in X_1\times\dots\times X_n$, where $X_i$ is a set of $D_i+1$ points in $\mathbb{Z}$.
\end{enumerate}

The first step also employs the univariate polynomial interpolation. In both steps, we need to know the value of $\bm{F}$ at some points. Since $\bm{F}(a_1,\dots,a_n)\in\mathbb{R}^m$ is in the kernel of the matrix $\bm{H}_{B_3}(f_G)(a_1,\dots,a_n)$ for any $(a_1,\dots,a_n)\in\mathbb{R}^n$, we can obtain information about the value of $\bm{F}(a_1,\dots,a_n)$ from the linear subspace $\ker\bm{H}_{B_3}(f_G)(a_1,\dots,a_n)$ of $\mathbb{R}^m$.

Let $(a_1,\dots,a_n)\in\mathbb{Z}_{>0}^n$. We compute $\bm{F}(a_1,\dots,a_n)\in\mathbb{Z}^m$ by the following steps.
\begin{enumerate}
    \item Confirm that the kernel $\ker\bm{H}_{B_3}(f_G)(a_1,\dots,a_n)$ is of dimension one.
    \item Take a non-zero vector $\bm{F}'$ from $\ker\bm{H}_{B_3}(f_G)(a_1,\dots,a_n)$.
    \item Find the coefficient $c\ne 0$ such that every component of the vector $c\bm{F}'$ is an integer, the greatest common divisor of the components of $c\bm{F}'$ is one, and $cF'_{i_0}>0$ for a predetermined index $i_0$.
    \item Lastly, this $c\bm{F}'$ will be the value of $\bm{F}(a_1,\dots,a_n)$.
\end{enumerate}

By the confirmation of the first step, the vectors $\bm{F}'$ and $\bm{F}(a_1,\dots,a_n)$ are parallel. In some counterexample candidates, this confirmation fails because the kernel of the third Hessian matrix is most likely of dimension two.

The third step is a kind of normalization. We expected the greatest common divisor of the components of $\bm{F}(a_1,\dots,a_n)$ to be one by choosing $a_i$ to be a prime power. We also guessed that $F_{i_0}(a_1,\dots,a_n)$ is always positive.

\subsubsection{The table of $\bm F$}
Our $\mathbb{R}$-basis $B_3$ of $A_3$ and the vector of polynomials $\bm F$ for the graph (\Cref{figure:SLP_3}) are on \Cref{table:B_3 and F}. This $\bm F$ has $90$ zeros and $76$ homogeneous polynomials of degree six. We found that $\bm{F}$ defined similarly for the graph shown in \Cref{figure:SLP_3 planar} also contains only $154$ zeros and $137$ homogeneous polynomials of degree six.

The sequence of maximum degrees is
\[(D_1,\dots,D_n)=(1,1,1,0,1,1,1,0,2,2,2,2,1)\]
and our predetermined index $i_0$ is two. Finally it was revealed that $F_2=2(x_1+x_5+x_{13})(x_2+x_6)x_{10}^2x_{11}x_{12}$, and thus $F_2(a_1,\dots,a_n)$ is always positive on $\mathbb{Z}_{>0}^n$.

To compress the table, we introduce some abbreviations:
{\allowdisplaybreaks
\begin{align*}
L_{a,b}&\coloneqq x_a+x_b\quad\left((a,b)=(2,6),(3,7),(9,10),(11,12)\right),\\
L_{1,5,13}&\coloneqq x_{1} + x_{5} + x_{13},\\
Q_{1}&\coloneqq x_{2} x_{9} + x_{6} x_{9} + x_{2} x_{10} + x_{6} x_{10} + x_{9} x_{10},\\
Q_{2}&\coloneqq x_{2} x_{9} + x_{6} x_{9} - x_{2} x_{10} - x_{6} x_{10} + x_{9} x_{10},\\
Q_{3}&\coloneqq x_{2} x_{9} + x_{6} x_{9} - x_{2} x_{10} - x_{6} x_{10} - x_{9} x_{10},\\
Q_{4}&\coloneqq x_{3} x_{11} + x_{7} x_{11} + x_{3} x_{12} + x_{7} x_{12} + x_{11} x_{12},\\
Q_{5}&\coloneqq x_{3} x_{11} + x_{7} x_{11} - x_{3} x_{12} - x_{7} x_{12} + x_{11} x_{12},\\
Q_{6}&\coloneqq x_{3} x_{11} + x_{7} x_{11} - x_{3} x_{12} - x_{7} x_{12} - x_{11} x_{12},\\
C_{1}&\coloneqq x_{2} x_{9}^{2} + x_{6} x_{9}^{2} + x_{9}^{2} x_{10} + x_{2} x_{10}^{2} + x_{6} x_{10}^{2} + x_{9} x_{10}^{2},\\
C_{2}&\coloneqq x_{2} x_{9}^{2} + x_{6} x_{9}^{2} + x_{9}^{2} x_{10} - x_{2} x_{10}^{2} - x_{6} x_{10}^{2} + x_{9} x_{10}^{2},\\
C_{3}&\coloneqq x_{2} x_{9}^{2} + x_{6} x_{9}^{2} - x_{9}^{2} x_{10} - x_{2} x_{10}^{2} - x_{6} x_{10}^{2} - x_{9} x_{10}^{2},\\
C_{4}&\coloneqq x_{3} x_{11}^{2} + x_{7} x_{11}^{2} + x_{11}^{2} x_{12} + x_{3} x_{12}^{2} + x_{7} x_{12}^{2} + x_{11} x_{12}^{2},\\
C_{5}&\coloneqq x_{3} x_{11}^{2} + x_{7} x_{11}^{2} + x_{11}^{2} x_{12} - x_{3} x_{12}^{2} - x_{7} x_{12}^{2} + x_{11} x_{12}^{2},\\
C_{6}&\coloneqq x_{3} x_{11}^{2} + x_{7} x_{11}^{2} - x_{11}^{2} x_{12} - x_{3} x_{12}^{2} - x_{7} x_{12}^{2} - x_{11} x_{12}^{2}.\\
\end{align*}
}

\begin{longtable}[c]{r|c|lcr|c|l}
\caption{The basis $B_3$ and the vector $\bm{F}$}\\
& $\alpha_i$ & $F_i$ & & & $\alpha_i$ & $F_i$\\ \cline{1-3}\cline{5-7} \endfirsthead
& $\alpha_i$ & $F_i$ & & & $\alpha_i$ & $F_i$\\ \cline{1-3}\cline{5-7} \endhead
1 & $(1,2,3)$ & $L_{1,5,13}(Q_{1}C_{4}-Q_{4}C_{1})$ & & 84 & $(2,11,13)$ & $0$\\
2 & $(1,2,4)$ & $2x_{12}x_{11}L_{2,6}L_{1,5,13}x_{10}^{2}$ & & 85 & $(2,12,13)$ & $0$\\
3 & $(1,2,5)$ & $0$ & & 86 & $(3,4,5)$ & $2x_{10}x_{9}L_{3,7}L_{1,5,13}x_{12}^{2}$\\
4 & $(1,2,7)$ & $L_{1,5,13}Q_{4}C_{1}$ & & 87 & $(3,4,6)$ & $2x_{13}L_{3,7}x_{12}^{2}Q_{1}$\\
5 & $(1,2,8)$ & $-2x_{12}x_{11}L_{2,6}L_{1,5,13}x_{10}^{2}$ & & 88 & $(3,4,7)$ & $-2x_{13}x_{10}x_{9}C_{4}$\\
6 & $(1,2,9)$ & $2x_{12}x_{11}L_{9,10}L_{1,5,13}Q_{2}$ & & 89 & $(3,4,9)$ & $-2x_{13}x_{10}x_{9}L_{9,10}Q_{4}$\\
7 & $(1,2,10)$ & $0$ & & 90 & $(3,4,11)$ & $0$\\
8 & $(1,2,11)$ & $L_{11,12}L_{1,5,13}Q_{6}Q_{1}$ & & 91 & $(3,5,6)$ & $L_{1,5,13}Q_{1}C_{4}$\\
9 & $(1,2,12)$ & $-L_{1,5,13}Q_{1}C_{5}$ & & 92 & $(3,5,8)$ & $-2x_{10}x_{9}L_{3,7}L_{1,5,13}x_{12}^{2}$\\
10 & $(1,3,4)$ & $-2x_{10}x_{9}L_{3,7}L_{1,5,13}x_{12}^{2}$ & & 93 & $(3,5,9)$ & $L_{9,10}L_{1,5,13}Q_{4}Q_{3}$\\
11 & $(1,3,5)$ & $0$ & & 94 & $(3,5,10)$ & $-L_{1,5,13}Q_{4}C_{2}$\\
12 & $(1,3,6)$ & $L_{1,5,13}Q_{1}C_{4}$ & & 95 & $(3,5,11)$ & $2L_{11,12}x_{10}x_{9}L_{1,5,13}Q_{5}$\\
13 & $(1,3,8)$ & $2x_{10}x_{9}L_{3,7}L_{1,5,13}x_{12}^{2}$ & & 96 & $(3,5,12)$ & $0$\\
14 & $(1,3,9)$ & $-L_{1,5,13}Q_{4}C_{3}$ & & 97 & $(3,6,8)$ & $-2x_{13}L_{3,7}x_{12}^{2}Q_{1}$\\
15 & $(1,3,10)$ & $L_{9,10}L_{1,5,13}Q_{4}Q_{2}$ & & 98 & $(3,6,9)$ & $-2x_{13}L_{9,10}Q_{4}Q_{2}$\\
16 & $(1,3,11)$ & $-2L_{11,12}x_{10}x_{9}L_{1,5,13}Q_{5}$ & & 99 & $(3,6,10)$ & $0$\\
17 & $(1,3,12)$ & $0$ & & 100 & $(3,6,11)$ & $2x_{13}L_{11,12}Q_{5}Q_{1}$\\
18 & $(1,4,5)$ & $0$ & & 101 & $(3,6,12)$ & $0$\\
19 & $(1,4,6)$ & $0$ & & 102 & $(3,6,13)$ & $0$\\
20 & $(1,4,7)$ & $0$ & & 103 & $(3,7,9)$ & $0$\\
21 & $(1,4,9)$ & $-2x_{12}x_{11}x_{10}x_{9}L_{9,10}L_{1,5,13}$ & & 104 & $(3,7,10)$ & $0$\\
22 & $(1,4,11)$ & $2x_{12}x_{11}L_{11,12}x_{10}x_{9}L_{1,5,13}$ & & 105 & $(3,7,11)$ & $-2x_{13}L_{11,12}x_{10}x_{9}Q_{5}$\\
23 & $(1,5,9)$ & $0$ & & 106 & $(3,7,12)$ & $2x_{13}L_{11,12}x_{10}x_{9}Q_{5}$\\
24 & $(1,5,10)$ & $0$ & & 107 & $(3,8,9)$ & $2x_{13}x_{10}x_{9}L_{9,10}Q_{4}$\\
25 & $(1,5,11)$ & $0$ & & 108 & $(3,8,11)$ & $0$\\
26 & $(1,5,12)$ & $0$ & & 109 & $(3,9,10)$ & $0$\\
27 & $(1,6,7)$ & $0$ & & 110 & $(3,9,11)$ & $0$\\
28 & $(1,6,8)$ & $0$ & & 111 & $(3,9,12)$ & $0$\\
29 & $(1,6,9)$ & $-2x_{12}x_{11}L_{9,10}L_{1,5,13}Q_{2}$ & & 112 & $(3,9,13)$ & $0$\\
30 & $(1,6,10)$ & $0$ & & 113 & $(3,10,11)$ & $0$\\
31 & $(1,6,11)$ & $-L_{1,5,13}Q_{1}C_{6}$ & & 114 & $(3,10,13)$ & $0$\\
32 & $(1,6,12)$ & $L_{11,12}L_{1,5,13}Q_{5}Q_{1}$ & & 115 & $(3,11,13)$ & $0$\\
33 & $(1,7,8)$ & $0$ & & 116 & $(4,5,9)$ & $2x_{12}x_{11}x_{10}x_{9}L_{9,10}L_{1,5,13}$\\
34 & $(1,7,9)$ & $L_{9,10}L_{1,5,13}Q_{4}Q_{3}$ & & 117 & $(4,5,11)$ & $-2x_{12}x_{11}L_{11,12}x_{10}x_{9}L_{1,5,13}$\\
35 & $(1,7,10)$ & $-L_{1,5,13}Q_{4}C_{2}$ & & 118 & $(4,6,9)$ & $0$\\
36 & $(1,7,11)$ & $2L_{11,12}x_{10}x_{9}L_{1,5,13}Q_{5}$ & & 119 & $(4,6,11)$ & $-2x_{13}x_{12}x_{11}L_{11,12}Q_{1}$\\
37 & $(1,7,12)$ & $0$ & & 120 & $(4,7,9)$ & $2x_{13}x_{10}x_{9}L_{9,10}Q_{4}$\\
38 & $(1,8,9)$ & $2x_{12}x_{11}x_{10}x_{9}L_{9,10}L_{1,5,13}$ & & 121 & $(4,7,11)$ & $0$\\
39 & $(1,8,11)$ & $-2x_{12}x_{11}L_{11,12}x_{10}x_{9}L_{1,5,13}$ & & 122 & $(4,8,9)$ & $-4x_{13}x_{12}x_{11}x_{10}x_{9}L_{9,10}$\\
40 & $(1,9,10)$ & $0$ & & 123 & $(4,8,11)$ & $4x_{13}x_{12}x_{11}L_{11,12}x_{10}x_{9}$\\
41 & $(1,9,11)$ & $0$ & & 124 & $(4,9,11)$ & $0$\\
42 & $(1,9,12)$ & $0$ & & 125 & $(5,6,9)$ & $2x_{12}x_{11}L_{9,10}L_{1,5,13}Q_{2}$\\
43 & $(1,10,11)$ & $0$ & & 126 & $(5,6,11)$ & $L_{11,12}L_{1,5,13}Q_{6}Q_{1}$\\
44 & $(1,10,12)$ & $0$ & & 127 & $(5,6,12)$ & $-L_{1,5,13}Q_{1}C_{5}$\\
45 & $(1,11,12)$ & $0$ & & 128 & $(5,7,9)$ & $-L_{1,5,13}Q_{4}C_{3}$\\
46 & $(2,3,4)$ & $2x_{13}(L_{2,6}x_{10}^{2}Q_{4}-L_{3,7}x_{12}^{2}Q_{1})$ & & 129 & $(5,7,10)$ & $L_{9,10}L_{1,5,13}Q_{4}Q_{2}$\\
47 & $(2,3,5)$ & $L_{1,5,13}(Q_{4}C_{1}-Q_{1}C_{4})$ & & 130 & $(5,7,11)$ & $-2L_{11,12}x_{10}x_{9}L_{1,5,13}Q_{5}$\\
48 & $(2,3,6)$ & $0$ & & 131 & $(5,8,9)$ & $-2x_{12}x_{11}x_{10}x_{9}L_{9,10}L_{1,5,13}$\\
49 & $(2,3,8)$ & $2x_{13}(L_{3,7}x_{12}^{2}Q_{1}-L_{2,6}x_{10}^{2}Q_{4})$ & & 132 & $(5,8,11)$ & $2x_{12}x_{11}L_{11,12}x_{10}x_{9}L_{1,5,13}$\\
50 & $(2,3,9)$ & $2x_{13}L_{9,10}Q_{4}Q_{2}$ & & 133 & $(5,9,10)$ & $0$\\
51 & $(2,3,10)$ & $0$ & & 134 & $(5,9,11)$ & $0$\\
52 & $(2,3,11)$ & $-2x_{13}L_{11,12}Q_{5}Q_{1}$ & & 135 & $(5,9,12)$ & $0$\\
53 & $(2,3,12)$ & $0$ & & 136 & $(5,10,11)$ & $0$\\
54 & $(2,3,13)$ & $0$ & & 137 & $(5,10,12)$ & $0$\\
55 & $(2,4,5)$ & $-2x_{12}x_{11}L_{2,6}L_{1,5,13}x_{10}^{2}$ & & 138 & $(5,11,12)$ & $0$\\
56 & $(2,4,6)$ & $2x_{13}x_{12}x_{11}C_{1}$ & & 139 & $(6,7,9)$ & $2x_{13}L_{9,10}Q_{4}Q_{2}$\\
57 & $(2,4,7)$ & $-2x_{13}L_{2,6}x_{10}^{2}Q_{4}$ & & 140 & $(6,7,11)$ & $-2x_{13}L_{11,12}Q_{5}Q_{1}$\\
58 & $(2,4,9)$ & $0$ & & 141 & $(6,8,11)$ & $2x_{13}x_{12}x_{11}L_{11,12}Q_{1}$\\
59 & $(2,4,11)$ & $2x_{13}x_{12}x_{11}L_{11,12}Q_{1}$ & & 142 & $(6,9,11)$ & $0$\\
60 & $(2,5,7)$ & $L_{1,5,13}Q_{4}C_{1}$ & & 143 & $(6,9,12)$ & $0$\\
61 & $(2,5,8)$ & $2x_{12}x_{11}L_{2,6}L_{1,5,13}x_{10}^{2}$ & & 144 & $(6,10,11)$ & $0$\\
62 & $(2,5,9)$ & $-2x_{12}x_{11}L_{9,10}L_{1,5,13}Q_{2}$ & & 145 & $(6,10,12)$ & $0$\\
63 & $(2,5,10)$ & $0$ & & 146 & $(6,11,12)$ & $0$\\
64 & $(2,5,11)$ & $-L_{1,5,13}Q_{1}C_{6}$ & & 147 & $(6,11,13)$ & $0$\\
65 & $(2,5,12)$ & $L_{11,12}L_{1,5,13}Q_{5}Q_{1}$ & & 148 & $(6,12,13)$ & $0$\\
66 & $(2,6,9)$ & $2x_{13}x_{12}x_{11}L_{9,10}Q_{2}$ & & 149 & $(7,8,9)$ & $-2x_{13}x_{10}x_{9}L_{9,10}Q_{4}$\\
67 & $(2,6,10)$ & $-2x_{13}x_{12}x_{11}L_{9,10}Q_{2}$ & & 150 & $(7,9,10)$ & $0$\\
68 & $(2,6,11)$ & $0$ & & 151 & $(7,9,11)$ & $0$\\
69 & $(2,6,12)$ & $0$ & & 152 & $(7,9,12)$ & $0$\\
70 & $(2,7,8)$ & $2x_{13}L_{2,6}x_{10}^{2}Q_{4}$ & & 153 & $(7,9,13)$ & $0$\\
71 & $(2,7,9)$ & $-2x_{13}L_{9,10}Q_{4}Q_{2}$ & & 154 & $(7,10,11)$ & $0$\\
72 & $(2,7,10)$ & $0$ & & 155 & $(7,10,13)$ & $0$\\
73 & $(2,7,11)$ & $2x_{13}L_{11,12}Q_{5}Q_{1}$ & & 156 & $(8,9,11)$ & $0$\\
74 & $(2,7,12)$ & $0$ & & 157 & $(9,10,11)$ & $0$\\
75 & $(2,7,13)$ & $0$ & & 158 & $(9,10,12)$ & $0$\\
76 & $(2,8,9)$ & $0$ & & 159 & $(9,10,13)$ & $0$\\
77 & $(2,8,11)$ & $-2x_{13}x_{12}x_{11}L_{11,12}Q_{1}$ & & 160 & $(9,11,12)$ & $0$\\
78 & $(2,9,11)$ & $0$ & & 161 & $(9,11,13)$ & $0$\\
79 & $(2,9,12)$ & $0$ & & 162 & $(9,12,13)$ & $0$\\
80 & $(2,9,13)$ & $0$ & & 163 & $(10,11,12)$ & $0$\\
81 & $(2,10,11)$ & $0$ & & 164 & $(10,11,13)$ & $0$\\
82 & $(2,10,12)$ & $0$ & & 165 & $(10,12,13)$ & $0$\\
83 & $(2,11,12)$ & $0$ & & 166 & $(11,12,13)$ & $0$
\label{table:B_3 and F}
\end{longtable}

\subsection{The list of Counterexample Candidates}\label{subsection:list of counterexample candidates}

\Cref{table:counterexample list} is the full list of counterexample candidates. The meaning of each column is as follows.
\begin{description}
    \item[Hilbert function] A triplet $(\dim_{\mathbb{R}}A_1,\dim_{\mathbb{R}}A_2,\dim_{\mathbb{R}}A_3)$. Note $\dim_{\mathbb{R}}A_0=1$ and $\dim_{\mathbb{R}}A_k=\dim_{\mathbb{R}}A_{7-k}$ for $k=4,5,6,7$. According to \cite{murai2021strictness}*{Theorem~2.5}, the number of edges in the corresponding simple graph $G$ is equal to $\dim_{\mathbb{R}}A_1$.
    \item[Nullity] Minimal dimension of $\ker\bm{H}_{B_3}(f_G)$ by substituting random numbers.
    \item[Planarity] Whether the graph is planar or not.
    \item[Edge set] The edge set, but in compressed format. Let the vertex set be $\{1,\dots,8\}$. The edge set of any simple graph is a subset of the set $\{(i,j)\mid 1\le i<j\le 8\}$. Here the pair $(i,j)$ is the edge between vertices $i$ and $j$. Enumerate this set in the lexicographic order, as $e_1=(1,2),\dots,e_{28}=(7,8)$. Then any edge set can be represented by a 28-character binary string, where each edge is included when the corresponding digit is one.

    For example, the first row ``0001111001 1110010100 11001000'' means
    \begin{multline*}
        \{e_4,e_5,e_6,e_7,e_{10},e_{11},e_{12},e_{13},e_{16},e_{18},e_{21},e_{22},e_{25}\}\\
        =\{(1,5),(1,6),(1,7),(1,8),(2,5),(2,6),(2,7),(2,8),\\(3,6),(3,8),(4,7),(4,8),(5,8)\}.
    \end{multline*}
\end{description}

\begin{longtable}[c]{r|c|c|c|l}
\caption{The list of counterexample candidates}\\
& Hilbert function & Nullity & Planarity & Edge set\\ \hline \endfirsthead
& Hilbert function & Nullity & Planarity & Edge set\\ \hline \endhead
1 & $(13, 70, 166)$ & $1$ & No & 0001111001 1110010100 11001000\\
2 & $(14, 74, 170)$ & $1$ & No & 0010110001 1100011100 01001111\\
3 & $(14, 78, 190)$ & $1$ & Yes & 0001111001 1110001100 11010010\\
4 & $(15, 82, 193)$ & $1$ & Yes & 0001110001 1010001100 11111110\\
5 & $(15, 89, 233)$ & $1$ & Yes & 0010011001 1110111100 11010010\\
6 & $(15, 90, 229)$ & $1$ & Yes & 0001110001 1010011100 11011110\\
7 & $(15, 91, 222)$ & $1$ & No & 0010111001 1100011101 10001011\\
8 & $(15, 93, 252)$ & $1$ & Yes & 0010111001 1010101100 11111000\\
9 & $(15, 95, 274)$ & $1$ & No & 0001110001 1010011101 11011010\\
10 & $(15, 95, 282)$ & $1$ & No & 0010111001 1110101100 11101000\\
11 & $(15, 97, 260)$ & $2$ & No & 0010111001 1110011100 11011000\\
12 & $(15, 97, 275)$ & $1$ & No & 0001110001 1100101101 11001011\\
13 & $(16, 93, 222)$ & $1$ & No & 0010111001 1000011100 11011111\\
14 & $(16, 94, 236)$ & $1$ & Yes & 0011100001 0110011111 00011111\\
15 & $(16, 95, 233)$ & $1$ & Yes & 0001101001 1000101101 11011111\\
16 & $(16, 98, 254)$ & $1$ & Yes & 0010111001 1010101100 11111001\\
17 & $(16, 100, 260)$ & $2$ & No & 0010111001 1110011100 11011001\\
18 & $(16, 101, 257)$ & $1$ & No & 0011111001 1010001111 00011110\\
19 & $(16, 101, 259)$ & $1$ & No & 0001111001 1100110100 11011110\\
20 & $(16, 101, 284)$ & $1$ & No & 0010111001 1110101100 11101001\\
21 & $(16, 103, 279)$ & $1$ & No & 0001101001 1010101101 11011101\\
22 & $(16, 103, 302)$ & $1$ & No & 0011101001 0110011111 01011100\\
23 & $(16, 104, 278)$ & $1$ & No & 0001101001 1010101101 11011110\\
24 & $(16, 104, 285)$ & $2$ & Yes & 0001110001 1010101101 11011110\\
25 & $(16, 104, 289)$ & $1$ & Yes & 0010111001 1100110100 11011110\\
26 & $(16, 106, 280)$ & $1$ & No & 0010111001 1110011100 11011010\\
27 & $(16, 106, 292)$ & $2$ & No & 0010111001 1110101100 11111000\\
28 & $(16, 106, 294)$ & $2$ & No & 0011111001 1100011101 10001011\\
29 & $(16, 106, 303)$ & $1$ & No & 0011110001 0110011111 10001011\\
30 & $(16, 106, 308)$ & $1$ & No & 0001111001 1100101101 11001011\\
31 & $(16, 106, 309)$ & $1$ & No & 0010111001 1110101100 11101100\\
32 & $(16, 106, 309)$ & $2$ & No & 0011101001 1100011101 10011011\\
33 & $(16, 106, 312)$ & $1$ & No & 0001111001 1100101101 11011010\\
34 & $(16, 108, 314)$ & $1$ & No & 0010111001 1110111100 11011000\\
35 & $(17, 103, 257)$ & $1$ & No & 0011111001 1110011100 11000111\\
36 & $(17, 103, 259)$ & $1$ & No & 0001111001 1000101101 11011111\\
37 & $(17, 109, 292)$ & $2$ & No & 0010111001 1110101100 11111001\\
38 & $(17, 109, 294)$ & $2$ & No & 0011111001 0110011100 11111001\\
39 & $(17, 110, 280)$ & $1$ & No & 0010111001 1110011100 11011011\\
40 & $(17, 110, 283)$ & $1$ & No & 0001101001 1010101101 11011111\\
41 & $(17, 110, 291)$ & $1$ & Yes & 0001110001 1010101101 11011111\\
42 & $(17, 110, 293)$ & $1$ & Yes & 0011101001 0110011111 00011111\\
43 & $(17, 111, 314)$ & $1$ & No & 0010111001 1110111100 11011001\\
44 & $(17, 112, 306)$ & $1$ & No & 0011101001 0110011111 01011110\\
45 & $(17, 112, 307)$ & $1$ & No & 0011101001 0110011111 01011101\\
46 & $(17, 112, 312)$ & $1$ & No & 0001111001 1100101101 11011011\\
47 & $(17, 113, 309)$ & $1$ & No & 0010111001 1110101100 11101101\\
48 & $(17, 113, 309)$ & $2$ & No & 0011111001 0110011100 11110011\\
49 & $(17, 114, 320)$ & $1$ & No & 0011100010 1110101101 11011110\\
50 & $(17, 115, 312)$ & $1$ & No & 0010111001 1110101100 11111010\\
51 & $(17, 115, 315)$ & $1$ & No & 0010111001 1110101100 11101110\\
52 & $(17, 115, 315)$ & $2$ & No & 0011111001 0110011100 11111010\\
53 & $(17, 115, 347)$ & $1$ & No & 0011111001 0110011111 01011100\\
54 & $(17, 116, 286)$ & $1$ & No & 0010111001 1110011100 11011110\\
55 & $(17, 116, 292)$ & $1$ & Yes & 0011011001 1110011100 11011110\\
56 & $(17, 116, 316)$ & $1$ & No & 0001111001 1010101101 11011110\\
57 & $(17, 116, 321)$ & $1$ & No & 0011111001 1110011100 11001110\\
58 & $(17, 116, 328)$ & $1$ & No & 0011111001 1110011101 10001011\\
59 & $(17, 116, 330)$ & $1$ & No & 0011111011 1110011100 10001110\\
60 & $(17, 116, 340)$ & $1$ & No & 0001111001 1110101101 11001011\\
61 & $(17, 116, 340)$ & $1$ & No & 0001111001 1110101101 11010011\\
62 & $(17, 116, 343)$ & $1$ & No & 0011111001 0110011111 10001011\\
63 & $(17, 116, 345)$ & $1$ & Yes & 0011110001 0110011111 01011110\\
64 & $(17, 118, 344)$ & $1$ & No & 0001111001 1110101101 11011010\\
65 & $(17, 118, 349)$ & $1$ & No & 0011111001 1110011101 01011100\\
66 & $(17, 120, 350)$ & $1$ & No & 0011111001 1110011100 11111000\\
67 & $(18, 116, 286)$ & $1$ & No & 0010111001 1110011100 11011111\\
68 & $(18, 116, 292)$ & $1$ & Yes & 0011011001 1110011100 11011111\\
69 & $(18, 117, 320)$ & $1$ & No & 0011100010 1110101101 11011111\\
70 & $(18, 119, 311)$ & $1$ & No & 0011101001 0110011111 01011111\\
71 & $(18, 119, 312)$ & $1$ & No & 0010111001 1110101100 11111011\\
72 & $(18, 119, 315)$ & $1$ & No & 0010111001 1110101100 11101111\\
73 & $(18, 119, 315)$ & $2$ & No & 0011111001 0110011100 11111011\\
74 & $(18, 119, 316)$ & $1$ & No & 0001111001 1010101101 11011111\\
75 & $(18, 119, 321)$ & $1$ & No & 0011111001 1110011100 11001111\\
76 & $(18, 119, 328)$ & $1$ & No & 0011111001 1110011101 10001111\\
77 & $(18, 119, 330)$ & $1$ & No & 0011111001 0110011111 00011111\\
78 & $(18, 122, 344)$ & $1$ & No & 0001111001 1110101101 11011011\\
79 & $(18, 122, 347)$ & $1$ & No & 0011111001 0110011111 01011101\\
80 & $(18, 122, 349)$ & $1$ & No & 0011111001 0110011111 11001011\\
81 & $(18, 122, 349)$ & $1$ & No & 0011111001 0110011111 11010011\\
82 & $(18, 123, 345)$ & $1$ & Yes & 0011110001 0110011111 01011111\\
83 & $(18, 124, 350)$ & $1$ & No & 0011111001 1110011100 11111001\\
84 & $(18, 124, 353)$ & $1$ & No & 0011111001 0110011111 11011010\\
85 & $(18, 125, 318)$ & $1$ & No & 0010111001 1110101100 11111110\\
86 & $(18, 125, 321)$ & $2$ & No & 0011111001 0110011100 11111110\\
87 & $(18, 125, 349)$ & $1$ & No & 0011111001 1110011101 10011011\\
88 & $(18, 125, 351)$ & $1$ & No & 0011111001 0110011111 01011110\\
89 & $(18, 126, 349)$ & $1$ & No & 0011111001 1110011101 01011110\\
90 & $(18, 127, 355)$ & $1$ & No & 0011111001 1110011100 11111010\\
91 & $(18, 128, 383)$ & $1$ & No & 0011111001 1110011110 11001110\\
92 & $(18, 128, 395)$ & $1$ & No & 0011111010 1110101101 01111100\\
93 & $(18, 129, 347)$ & $2$ & No & 0011111001 1110011100 11011110\\
94 & $(18, 129, 348)$ & $1$ & No & 0001111001 1110101101 11011110\\
95 & $(18, 129, 370)$ & $1$ & No & 0011011001 1110011101 11011110\\
96 & $(18, 129, 399)$ & $1$ & No & 0011111011 1110011110 10001110\\
97 & $(18, 131, 393)$ & $1$ & No & 0011111011 1110111100 10001011\\
98 & $(18, 131, 394)$ & $1$ & No & 0011111010 1110101101 11111000\\
99 & $(18, 131, 398)$ & $1$ & No & 0011111001 1110111100 11110010\\
100 & $(19, 125, 318)$ & $1$ & No & 0010111001 1110101100 11111111\\
101 & $(19, 125, 321)$ & $2$ & No & 0011111001 0110011100 11111111\\
102 & $(19, 128, 353)$ & $1$ & No & 0011111001 0110011111 11011011\\
103 & $(19, 129, 347)$ & $2$ & No & 0011111001 1110011100 11011111\\
104 & $(19, 129, 348)$ & $1$ & No & 0001111001 1110101101 11011111\\
105 & $(19, 129, 349)$ & $1$ & No & 0011111001 1110011101 01011111\\
106 & $(19, 129, 351)$ & $1$ & No & 0011111001 0110011111 01011111\\
107 & $(19, 129, 370)$ & $1$ & No & 0011011001 1110011101 11011111\\
108 & $(19, 132, 355)$ & $1$ & No & 0011111001 1110011100 11111011\\
109 & $(19, 132, 383)$ & $1$ & No & 0011111001 1110011110 11001111\\
110 & $(19, 132, 399)$ & $1$ & No & 0101111011 0110110001 11011111\\
111 & $(19, 135, 355)$ & $1$ & No & 0011111001 1110011100 11111110\\
112 & $(19, 135, 357)$ & $1$ & No & 0011111001 0110011111 11011110\\
113 & $(19, 135, 393)$ & $1$ & No & 0011111001 1110111100 11111001\\
114 & $(19, 135, 394)$ & $1$ & No & 0011111010 1110101101 11111001\\
115 & $(19, 135, 395)$ & $1$ & No & 0011111010 1110101101 01111110\\
116 & $(19, 135, 403)$ & $1$ & No & 0101111011 1110101101 10101011\\
117 & $(19, 136, 395)$ & $1$ & No & 0011111010 1110101101 01111101\\
118 & $(19, 136, 399)$ & $1$ & No & 0011111011 1110011110 10001111\\
119 & $(19, 138, 398)$ & $1$ & No & 0011111001 1110111100 11111010\\
120 & $(19, 138, 399)$ & $1$ & No & 0011111010 1110101101 11111010\\
121 & $(19, 138, 402)$ & $1$ & No & 0011111011 1110111100 10101011\\
122 & $(19, 138, 407)$ & $1$ & No & 0110111011 1110111101 11001100\\
123 & $(19, 139, 387)$ & $1$ & No & 0011111001 1110011110 11011110\\
124 & $(19, 139, 391)$ & $1$ & No & 0011111011 0110011101 11011110\\
125 & $(19, 139, 403)$ & $1$ & No & 0011111011 1110011110 11001110\\
126 & $(19, 140, 398)$ & $1$ & No & 0011111001 1110111100 11110011\\
127 & $(19, 143, 398)$ & $1$ & No & 0011111011 1110011100 11011110\\
128 & $(20, 135, 355)$ & $1$ & No & 0011111001 1110011100 11111111\\
129 & $(20, 135, 357)$ & $1$ & No & 0011111001 0110011111 11011111\\
130 & $(20, 139, 387)$ & $1$ & No & 0011111001 1110011110 11011111\\
131 & $(20, 139, 391)$ & $1$ & No & 0011111011 0110011101 11011111\\
132 & $(20, 139, 395)$ & $1$ & No & 0011111010 1110101101 01111111\\
133 & $(20, 139, 399)$ & $1$ & No & 0101110011 1110110100 11111111\\
134 & $(20, 142, 402)$ & $1$ & No & 0101111011 1110111100 11111001\\
135 & $(20, 142, 403)$ & $1$ & No & 0101111011 1110110100 11111110\\
136 & $(20, 143, 398)$ & $1$ & No & 0011111001 1110111100 11111011\\
137 & $(20, 143, 398)$ & $1$ & No & 0011111011 1110011100 11011111\\
138 & $(20, 143, 399)$ & $1$ & No & 0011111010 1110101101 11111011\\
139 & $(20, 143, 403)$ & $1$ & No & 0011111011 1110011110 11001111\\
140 & $(20, 143, 403)$ & $1$ & No & 0101111011 1110110100 11111101\\
141 & $(20, 145, 407)$ & $1$ & No & 0101111011 1110111100 11111010\\
142 & $(20, 146, 398)$ & $1$ & No & 0011111001 1110111100 11111110\\
143 & $(20, 146, 399)$ & $1$ & No & 0011111010 1110101101 11111110\\
144 & $(20, 147, 407)$ & $1$ & No & 0101111011 1110111100 11110011\\
145 & $(20, 150, 407)$ & $1$ & No & 0011111011 1110011110 11011110\\
146 & $(21, 146, 398)$ & $1$ & No & 0011111001 1110111100 11111111\\
147 & $(21, 146, 399)$ & $1$ & No & 0011111010 1110101101 11111111\\
148 & $(21, 146, 403)$ & $1$ & No & 0101111011 1110110100 11111111\\
149 & $(21, 150, 407)$ & $1$ & No & 0011111011 1110011110 11011111\\
150 & $(21, 150, 407)$ & $1$ & No & 0101111011 1110111100 11111011\\
151 & $(21, 153, 407)$ & $1$ & No & 0101111011 1110111100 11111110\\
152 & $(22, 153, 407)$ & $1$ & No & 0101111011 1110111100 11111111
\label{table:counterexample list}
\end{longtable}

\Cref{figure:SLP_3,figure:SLP_3 planar}, which we have verified as counterexamples, are the first and 41st counterexample candidates, respectively.

The 152nd counterexample candidate (\Cref{figure:maximal candidate}) has the largest number of edges. It is constructed by adding two black vertices and seven dashed edges to the complete graph of six vertices. It contains all listed graphs as subgraphs, e.g., it contains the first counterexample candidate as shown in \Cref{figure:contain}.

\begin{figure}[tb]
    \begin{minipage}{.49\textwidth}
        \centering
        \begin{tikzpicture}[scale=0.7]
            \node[draw,circle,line width=1pt,fill=black](a) at (0,6){};
            \node[draw,circle,line width=1pt,fill=black](b) at (0,4){};
            \node[draw,circle,line width=1pt](1) at (2,2){};
            \node[draw,circle,line width=1pt](2) at (4,0){};
            \node[draw,circle,line width=1pt](3) at (2,-2){};
            \node[draw,circle,line width=1pt](4) at (-2,-2){};
            \node[draw,circle,line width=1pt](5) at (-4,0){};
            \node[draw,circle,line width=1pt](6) at (-2,2){};
    
            \draw[line width=1pt,dashed](a)--(b);
            \draw[line width=1pt,dashed](a)to[out=0,in=90](2);
            \draw[line width=1pt,dashed](a)to[out=180,in=90](5);
            \draw[line width=1pt,dashed](b)--(1);
            \draw[line width=1pt,dashed](b)to[out=0,in=110](2);
            \draw[line width=1pt,dashed](b)to[out=180,in=70](5);
            \draw[line width=1pt,dashed](b)--(6);
            \draw[line width=1pt](1)--(2);
            \draw[line width=1pt](1)--(3);
            \draw[line width=1pt](1)--(4);
            \draw[line width=1pt](1)--(5);
            \draw[line width=1pt](1)--(6);
            \draw[line width=1pt](2)--(3);
            \draw[line width=1pt](2)--(4);
            \draw[line width=1pt](2)--(5);
            \draw[line width=1pt](2)--(6);
            \draw[line width=1pt](3)--(4);
            \draw[line width=1pt](3)--(5);
            \draw[line width=1pt](3)--(6);
            \draw[line width=1pt](4)--(5);
            \draw[line width=1pt](4)--(6);
            \draw[line width=1pt](5)--(6);
        \end{tikzpicture}
        \caption{The candidate graph with the largest number of edges}
        \label{figure:maximal candidate}
    \end{minipage}
    \begin{minipage}{.49\textwidth}
        \centering
        \begin{tikzpicture}[scale=0.7]
            \node[draw,circle,line width=1pt,fill=black](5) at (0,2){};
            \node[draw,circle,line width=1pt](2) at (4,0){};
            \node[draw,circle,line width=1pt,fill=black](8) at (0,0){};
            \node[draw,circle,line width=1pt](1) at (-4,0){};
            \node[draw,circle,line width=1pt](4) at (2,-2){};
            \node[draw,circle,line width=1pt](7) at (4,-4){};
            \node[draw,circle,line width=1pt](3) at (-2,-2){};
            \node[draw,circle,line width=1pt](6) at (-4,-4){};
    
            \draw[line width=1pt,dashed](1)--(5);
            \draw[line width=1pt](1)--(6);
            \draw[line width=1pt](1)--(7);
            \draw[line width=1pt,dashed](1)--(8);
            \draw[line width=1pt,dashed](2)--(5);
            \draw[line width=1pt](2)--(6);
            \draw[line width=1pt](2)--(7);
            \draw[line width=1pt,dashed](2)--(8);
            \draw[line width=1pt](3)--(6);
            \draw[line width=1pt,dashed](3)--(8);
            \draw[line width=1pt](4)--(7);
            \draw[line width=1pt,dashed](4)--(8);
            \draw[line width=1pt,dashed](5)--(8);
        \end{tikzpicture}
        \caption{The graph in \Cref{figure:maximal candidate} contains the graph in \Cref{figure:SLP_3}}
        \label{figure:contain}
    \end{minipage}
\end{figure}

\section{Partial Failure of the $\SLP_2$}\label{section:SLP_2}

First of all, every graph of eight or fewer vertices has the $\SLP_2$. However, only one graph (\Cref{figure:SLP_2}) does not have the $\SLP_2$ with fixed element $\ell=\partial_1+\dots+\partial_n$. This graph also does not have the $\SLP_3$ with the same element $\ell$, but has the strong Lefschetz property with other elements.

This $\ell$ is typically one of the Lefschetz elements of the strong Lefschetz property. A previous work~\cite{yazawa2021eigenvalues} showed that the complete or complete bipartite graph has the $\SLP_1$ with this element $\ell$. Besides, $\det\bm{H}_{B_k}(f)(1,\dots,1)$ is calculated in~\cite{maeno2016sperner}. For any set of edges $I$, $\left(\left(\prod_{i\in I}\partial_i\right)f_G\right)(1,\dots,1)$ is equal to the number of spanning trees in $G$ which contains all of edges in $I$.

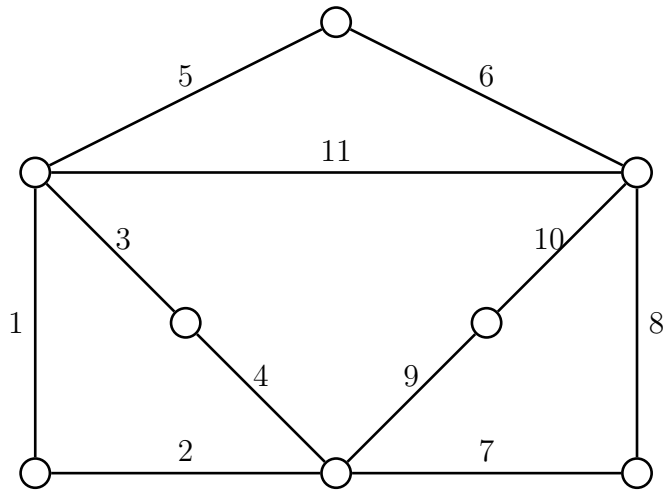
\begin{figure}[tb]
    \centering
    \begin{tikzpicture}
        \node[draw,circle,line width=1pt](3) at (0,2){};
        \node[draw,circle,line width=1pt](8) at (4,0){};
        \node[draw,circle,line width=1pt](6) at (-4,0){};
        \node[draw,circle,line width=1pt](5) at (2,-2){};
        \node[draw,circle,line width=1pt](4) at (4,-4){};
        \node[draw,circle,line width=1pt](2) at (-2,-2){};
        \node[draw,circle,line width=1pt](1) at (-4,-4){};
        \node[draw,circle,line width=1pt](7) at (0,-4){};

        \draw[line width=1pt](1)--(6)node[pos=0.5,left]{$1$};
        \draw[line width=1pt](1)--(7)node[pos=0.5,above]{$2$};
        \draw[line width=1pt](2)--(6)node[pos=0.4,above]{$3$};
        \draw[line width=1pt](2)--(7)node[pos=0.5,above]{$4$};
        \draw[line width=1pt](3)--(6)node[pos=0.5,above]{$5$};
        \draw[line width=1pt](3)--(8)node[pos=0.5,above]{$6$};
        \draw[line width=1pt](4)--(7)node[pos=0.5,above]{$7$};
        \draw[line width=1pt](4)--(8)node[pos=0.5,right]{$8$};
        \draw[line width=1pt](5)--(7)node[pos=0.5,above]{$9$};
        \draw[line width=1pt](5)--(8)node[pos=0.4,above]{$10$};
        \draw[line width=1pt](6)--(8)node[pos=0.5,above]{$11$};
    \end{tikzpicture}
    \caption{The graph does not have the $\SLP_2$ with the element $\ell=\partial_1+\dots+\partial_n$}
    \label{figure:SLP_2}
\end{figure}

Let $G$ be the graph shown in \Cref{figure:SLP_2}. The number of variables $n=11$, the socle degree $d=7$, the minimal number of generators $\mu(\Ann_Q(f_G))=42$, and the Hilbert function $(\dim_{\mathbb{R}}A_i)_{i=0}^d=(1,11,51,112,112,51,11,1)$. As in \Cref{section:SLP_3}, we construct an $\mathbb{R}$-basis $B_2$ of $A_2$ by the same method and find a non-zero vector $\bm{F}=(F_1,\dots,F_{51})^\mathrm{T}\in\mathbb{R}^{51}$ such that $\bm{H}_{B_2}(f_G)(1,\dots,1)\bm{F}=\bm{0}$.

The vector $\bm{F}$ has only $16$ non-zero components. Our $B_2$ and $\bm{F}$ are on \Cref{table:B_2 and F}. Empty cells mean that the corresponding monomials are not in $B_2$: $\partial_2\partial_4$, $\partial_5\partial_{11}$, $\partial_6\partial_{11}$, and $\partial_8\partial_{10}$. Non-empty cells contain corresponding $F_i$, e.g., $F_1=10$ for $\alpha_1=\partial_1\partial_2$. \Cref{table:B_2 and F} shows symmetries of the squares $\{\partial_1,\dots,\partial_9\}\times\{\partial_2,\dots,\partial_{10}\}$ and of the rightmost column.
\begin{table}[tb]
\centering
\caption{The basis $B_2$ and the vector $\bm{F}$}
\begin{tabular}{r|rrrrrrrrrrr}
& $\partial_{1}$ & $\partial_{2}$ & $\partial_{3}$ & $\partial_{4}$ & $\partial_{5}$ & $\partial_{6}$ & $\partial_{7}$ & $\partial_{8}$ & $\partial_{9}$ & $\partial_{10}$ & $\partial_{11}$\\ \hline
$\partial_{1}$ & \multicolumn{1}{r|}{}& \multicolumn{1}{r|}{$10$}& \multicolumn{1}{r|}{$0$}& \multicolumn{1}{r|}{$-4$}& \multicolumn{1}{r|}{$0$}& \multicolumn{1}{r|}{$0$}& \multicolumn{1}{r|}{$0$}& \multicolumn{1}{r|}{$0$}& \multicolumn{1}{r|}{$0$}& \multicolumn{1}{r|}{$0$}& \multicolumn{1}{r|}{$-3$}\\ \cline{3-12}
$\partial_{2}$ & & \multicolumn{1}{r|}{}& \multicolumn{1}{r|}{$-4$}& \multicolumn{1}{r|}{}& \multicolumn{1}{r|}{$0$}& \multicolumn{1}{r|}{$0$}& \multicolumn{1}{r|}{$0$}& \multicolumn{1}{r|}{$0$}& \multicolumn{1}{r|}{$0$}& \multicolumn{1}{r|}{$0$}& \multicolumn{1}{r|}{$-3$}\\ \cline{4-12}
$\partial_{3}$ & & & \multicolumn{1}{r|}{}& \multicolumn{1}{r|}{$10$}& \multicolumn{1}{r|}{$0$}& \multicolumn{1}{r|}{$0$}& \multicolumn{1}{r|}{$0$}& \multicolumn{1}{r|}{$0$}& \multicolumn{1}{r|}{$0$}& \multicolumn{1}{r|}{$0$}& \multicolumn{1}{r|}{$-3$}\\ \cline{5-12}
$\partial_{4}$ & & & & \multicolumn{1}{r|}{}& \multicolumn{1}{r|}{$0$}& \multicolumn{1}{r|}{$0$}& \multicolumn{1}{r|}{$0$}& \multicolumn{1}{r|}{$0$}& \multicolumn{1}{r|}{$0$}& \multicolumn{1}{r|}{$0$}& \multicolumn{1}{r|}{$-3$}\\ \cline{6-12}
$\partial_{5}$ & & & & & \multicolumn{1}{r|}{}& \multicolumn{1}{r|}{$0$}& \multicolumn{1}{r|}{$0$}& \multicolumn{1}{r|}{$0$}& \multicolumn{1}{r|}{$0$}& \multicolumn{1}{r|}{$0$}& \multicolumn{1}{r|}{}\\ \cline{7-12}
$\partial_{6}$ & & & & & & \multicolumn{1}{r|}{}& \multicolumn{1}{r|}{$0$}& \multicolumn{1}{r|}{$0$}& \multicolumn{1}{r|}{$0$}& \multicolumn{1}{r|}{$0$}& \multicolumn{1}{r|}{}\\ \cline{8-12}
$\partial_{7}$ & & & & & & & \multicolumn{1}{r|}{}& \multicolumn{1}{r|}{$-10$}& \multicolumn{1}{r|}{$0$}& \multicolumn{1}{r|}{$4$}& \multicolumn{1}{r|}{$3$}\\ \cline{9-12}
$\partial_{8}$ & & & & & & & & \multicolumn{1}{r|}{}& \multicolumn{1}{r|}{$4$}& \multicolumn{1}{r|}{}& \multicolumn{1}{r|}{$3$}\\ \cline{10-12}
$\partial_{9}$ & & & & & & & & & \multicolumn{1}{r|}{}& \multicolumn{1}{r|}{$-10$}& \multicolumn{1}{r|}{$3$}\\ \cline{11-12}
$\partial_{10}$ & & & & & & & & & & \multicolumn{1}{r|}{}& \multicolumn{1}{r|}{$3$}\\ \cline{12-12}
$\partial_{11}$ & & & & & & & & & & & \multicolumn{1}{r|}{}
\end{tabular}
\label{table:B_2 and F}
\end{table}

\section*{Acknowledgments}

I would like to thank Toshiaki Maeno and Yasuhide Numata for their useful discussions and comments. I am grateful to Akihito Wachi for his remarks about \Cref{remark: biconnected components} and for his recalculation. This work was supported by JST SPRING, Grant Number JPMJSP2114.


\end{document}